\documentclass{article}

\usepackage{amsmath, amsthm, amssymb, mathrsfs, amscd, amsthm, amscd,amsfonts,calligra,mathrsfs,adjustbox,lipsum}
\usepackage{indentfirst,url}
\usepackage[all]{xy}
\usepackage{url}
\usepackage{tikz}
\usepackage{tikz-cd}
\usepackage[colorlinks,plainpages]{hyperref}
\usepackage{paracol}
\usepackage{float}
\usetikzlibrary{arrows,decorations.markings}
\usepackage{mathtools}
\usepackage{caption}
\usepackage[margin=75pt]{geometry}

\hypersetup{
	colorlinks=true,
	linkcolor=blue,
	filecolor=magenta,      
	urlcolor=cyan,
}
\newtheorem*{Acknowledgments}{ACKNOWLEDGMENTS}

\makeatletter
\def\@endtheorem{\endtrivlist\@endpefalse } % Remove the end-of-theorem box
\makeatother

\makeatletter
\def\th@plainitalic{%
  \thm@notefont{}% same as heading font
  \itshape % body font
}
\makeatother

\theoremstyle{plainitalic}

\newtheorem{Theorem}{Theorem}
\newtheorem{Corollary}{Corollary}
\newtheorem{Proposition}{Proposition}
\newtheorem{Lemma}{Lemma}

\newtheorem{Remark}{Remark}

\begin{document}

% Enter full title and short title for running headers
\title{Infinitely many new sequences of surfaces  of  general type with maximal Picard number converging to the Severi line}

% Author name(s)
\author{Nguyen Bin and Vicente Lorenzo}

\newcommand{\BinAddresses}{{% additional braces for segregating \footnotesize
		\bigskip
		\footnotesize
        \text{Nguyen Bin,}\par\nopagebreak	
		\text{Department of Mathematics and Statistics,}\par\nopagebreak	
		\text{Quy Nhon University,}\par\nopagebreak	
            \text{170 An Duong Vuong Street, Quy Nhon,}\par\nopagebreak
		\text{Vietnam.}\par\nopagebreak		
		\textit{E-mail address}: \texttt{nguyenbin@qnu.edu.vn}
		
}}

\newcommand{\VicenteAddresses}{{% additional braces for segregating \footnotesize
		\bigskip
		\footnotesize
        \text{Vicente Lorenzo,}\par\nopagebreak	
		\text{E.T.S. de Ingeniería y Sist. de Telecom.,}\par\nopagebreak	
		\text{Universidad Politécnica de Madrid,}\par\nopagebreak	
            \text{Calle Nikola Tesla, Madrid,}\par\nopagebreak	
		\text{Spain.}\par\nopagebreak		
		\textit{E-mail address}: \texttt{vlorenzogarcia@gmail.com}
		
}}

\newcommand\blfootnote[1]{%
	\begingroup
	\renewcommand\thefootnote{}\footnote{#1}%
	\addtocounter{footnote}{-1}%
	\endgroup
}

\newcommand{\MSC}{\textbf{Mathematics Subject Classification (2010):}}

\newcommand{\Key}{\textbf{Key words:}}

\maketitle

 \begin{abstract}
 \noindent
Examples of algebraic surfaces of general type with maximal Picard number are not abundant in the literature. Moreover, most known examples either possess low invariants, lie near the Noether line $K^2=2\chi-6$ or are somewhat  scattered. A notable exception is Persson's sequence of double covers of the projective plane with maximal Picard number, whose invariants converge to the Severi line $K^2=4\chi$. This note is devoted to the construction of infinitely many new sequences of surfaces of general type with maximal Picard number whose invariants converge to the Severi line.
 \end{abstract}

 \blfootnote{\MSC{ 14J29}.}
\blfootnote{\Key{ Picard number, Surfaces of general type, Abelian covers.}}

\section{Introduction} \label{Introduction}

The self-intersection of the canonical class $K^2_X$ and the holomorphic Euler characteristic $\chi(\mathcal{O}_X)$ of an algebraic surface $X$ are two of its main numerical invariants. If $X$ is a minimal surface of general type  defined over the complex numbers $\mathbb{C}$, it is well-known that $(K_X^2,\chi(\mathcal{O}_X))$ is a pair of strictly positive integers satisfying both Noether's inequality $K^2_X\geq 2\chi(\mathcal{O}_X)-6$ and the Bogomolov-Miyaoka-Yau inequality $K^2_X\leq 9\chi(\mathcal{O}_X)$ (cf. 
\cite[Chapter VII]{MR2030225}). More challenging is the inverse problem of determining whether given an admissible pair, i.e. a pair of strictly positive integers $(K^2, \chi)$ satisfying both Noether's inequality and the Bogomolov-Miyaoka-Yau inequality,  there exists a minimal surface of general type $X$ such that $K^2_X=K^2$ and $\chi(\mathcal{O}_X)=\chi$. The first systematic approach to this problem, known as the geographical question, was taken by Persson \cite{MR631426}, who filled in most of the region $2\chi-6\leq K^2\leq 8\chi$.
Chen \cite{MR0884652}, \cite{MR1091940}
then filled in some of the remaining gaps, extending coverage to a substantial portion of the region $8\chi< K^2 <9\chi$. 
Despite the presence of some missing examples, the current consensus is that these instances primarily stem from technical factors rather than indicating any inherent issues.

Given the prevailing agreement that the original geographical question can be given an affirmative answer, numerous authors directed their attention towards exploring the geography of surfaces of general type with special features such as  genus $2$
fibrations \cite{MR0555712}, simply-connectedness \cite{MR631426}, $2$-divisibility of the canonical
class \cite{MR1404912}, global $1$-forms \cite{MR2931875}, $\mathbb{Z}_2^2$-actions \cite{MR4626843}, etc.
% more examples should be included?
Specially intriguing is the case of the geography of surfaces of general type with maximal Picard number. 

The Picard number $\rho(X)$ of a smooth projective surface $X$ is the rank of its Neron-Severi group $NS(X)$ which, defined as the  group of divisors of $X$ modulo numerical equivalence, is finitely generated. The Picard number of $X$ is bounded above by the Hodge number $h^{1,1}(X)=\text{dim } H^1(X,\Omega^1_X)$ and  $X$ is said to have maximal Picard number if $\rho(X)=h^{1,1}(X)$. Although there is no a priori reason to believe the geography of minimal surfaces of general type should not be highly populated by surfaces with maximal Picard number, the examples of such surfaces are scarce in the literature (see \cite{MR3322784} for an overview on the subject). In summary:
\begin{itemize}
    \item Surfaces $X$ with geometric genus $p_g(X)=0$ have maximal Picard number.
    \item The first non-trivial examples of algebraic surfaces of general type with maximal Picard number were published by Persson \cite{MR661198}. On the one hand, if we denote by $\mathfrak{M}_{K^2,\chi}$ Gieseker's moduli space \cite{MR498596} of canonical models of surfaces of general type with self-intersection of the canonical class $K^2$ and holomorphic Euler characteristic $\chi$, Persson \cite[Theorem 1]{MR661198} proved that given an admissible pair $(K^2,\chi)$ such that $K^2=2\chi-6$ and $\chi\not\equiv 0 \text{ mod } 6$, then all the connected components of $\mathfrak{M}_{K^2,\chi}$ contain canonical models whose minimal resolution has maximal Picard number. On the other hand, Persson \cite[Theorem 3]{MR661198} proved that for every integer $n\geq 4$, if we denote $K^2=2(n-3)^2$ and $ \chi=\frac{1}{2}(n-1)(n-2)+1$, then $\mathfrak{M}_{K^2,\chi}$ contains canonical models whose minimal resolution has maximal Picard number.
    \item  Examples of surfaces of general type with maximal Picard number and low geometric genus can be found in 
\cite[6.2, 6.3]{MR3322784},
\cite{MR3460339}, \cite{MR3683423}.
\item Some scattered examples can be found in \cite[6.5]{MR3322784}, \cite{MR2855811}. 
\item Further examples can be found in \cite{2014arXiv1406.2143A}, \cite{2016arXiv161100470S}.
\item The authors \cite{MR4761778}
proved that $\mathfrak{M}_{K^2,\chi}$ contains canonical models whose minimal resolution has maximal Picard number for every admissible pair such that $K^2\leq \frac{5}{2}\chi-11$. In addition, some scattered examples one above the Noether line were provided.
\end{itemize}

 The main result of \cite{MR4761778} 
draws inspiration from \cite[Theorem 1]{MR661198}. Whereas Persson obtained the examples of \cite[Theorem 1]{MR661198} as double covers of rational surfaces whose branch locus has a very particular configuration, the authors \cite{MR4761778} obtained theirs as bidouble covers of rational surfaces. Similarly, this paper introduces new families of surfaces of general type with maximal Picard number drawing inspiration from  \cite[Theorem 3]{MR661198}, where surfaces of general type with maximal Picard number were obtained as double covers of the projective plane $\mathbb{P}^2=\text{Proj}(\mathbb{C}[X_0,X_1,X_2])$ whose branch locus derives from the curve: 
 \begin{equation*}
C=[(X_0^n+X_1^n+X_2^n)^2-4\cdot ((X_0X_1)^n+(X_0X_2)^n+(X_1X_2)^n)=0].
 \end{equation*}

 Considering bidouble covers of the projective plane whose branch locus derives from the curve $C$, one can get the following:

 \begin{Theorem}\label{main_theorem}
    Given  $n \in \mathbb{Z}_{\geq 2}$ let us define 
    \begin{equation*}
        K^2=4n^2-12n+9 \qquad \text{and} \qquad \chi=n^2-n+1.
    \end{equation*}
    Then $\mathfrak{M}_{K^2, \chi}$ contains canonical models $X_{n}$ whose minimal resolution has maximal Picard number.
\end{Theorem}

Blowing-up $\mathbb{P}^2$ at a point, taking cyclic covers branched along two fibers of the blown-up surface and then considering a bidouble cover whose branch locus derives from $C$, one can get the following:

\begin{Theorem}\label{main_theorem_1}
    Given $m\in\mathbb{Z}_{\geq 3}$ and $ n\in 2\cdot \mathbb{Z}_{\geq 1}$   let us define 
    \begin{equation*}
        K^2=4mn^2-4(m+2)n+8 \qquad \text{and} \qquad \chi=mn^2-n+1.
    \end{equation*}
    Then $\mathfrak{M}_{K^2, \chi}$ contains canonical models $X_{m,n}$ whose minimal resolution has maximal Picard number.
\end{Theorem}

Finally, blowing-up $\mathbb{P}^2$ at a point, taking cyclic covers branched along two fibers of the blown-up surface and then considering a  double cover whose branch locus derives from $C$, one can get the following:

\begin{Theorem}\label{main_theorem_2}
    Given $m\in \mathbb{Z}_{\geq 2}$ and $ n\in 2\cdot \mathbb{Z}_{\geq 2}$ let us define 
    \begin{equation*}
        K^2=2mn^2-4(m+1)n+8 \qquad \text{and} \qquad \chi=\frac{1}{2}mn(n -1)+1.
    \end{equation*}
    Then $\mathfrak{M}_{K^2, \chi}$ contains canonical models $X_{m,n}$ whose minimal resolution has maximal Picard number.
\end{Theorem}

%{
%\color{violet}
%\begin{Theorem}\label{main_theorem_2'}
%    Given integers $m\geq 2, n \geq 2$, let us define 
%    \begin{equation*}
%       K^2=2mn^2-4n -2m+4 \qquad \text{and} \qquad \chi=\frac{m}{2}n^2+\frac{m}{2}n+1.
%    \end{equation*}
%    Then $\mathfrak{M}_{K^2, \chi}$ contains canonical models $X_{k,n}$ whose minimal resolution has maximal Picard number.
%\end{Theorem}

%\begin{Remark}
%    In the case $n = 2$, surfaces in Theorem \ref{main_theorem_1} belong to the line $K^2 = 2\chi - 6$ with $K^2 = 8m -8$ and $\chi = 4m -1$, $m\geq 2$.
%\end{Remark}

The region of the plane $(K^2, \chi)$ covered by the pairs described in Theorem \ref{main_theorem}, Theorem \ref{main_theorem_1} and Theorem \ref{main_theorem_2} may seem unstructured at a first glance, but paying attention to the slopes $\mu=K^2/\chi$ one can infer the following:

\begin{figure}
        \centering
        \includegraphics[width=\linewidth]{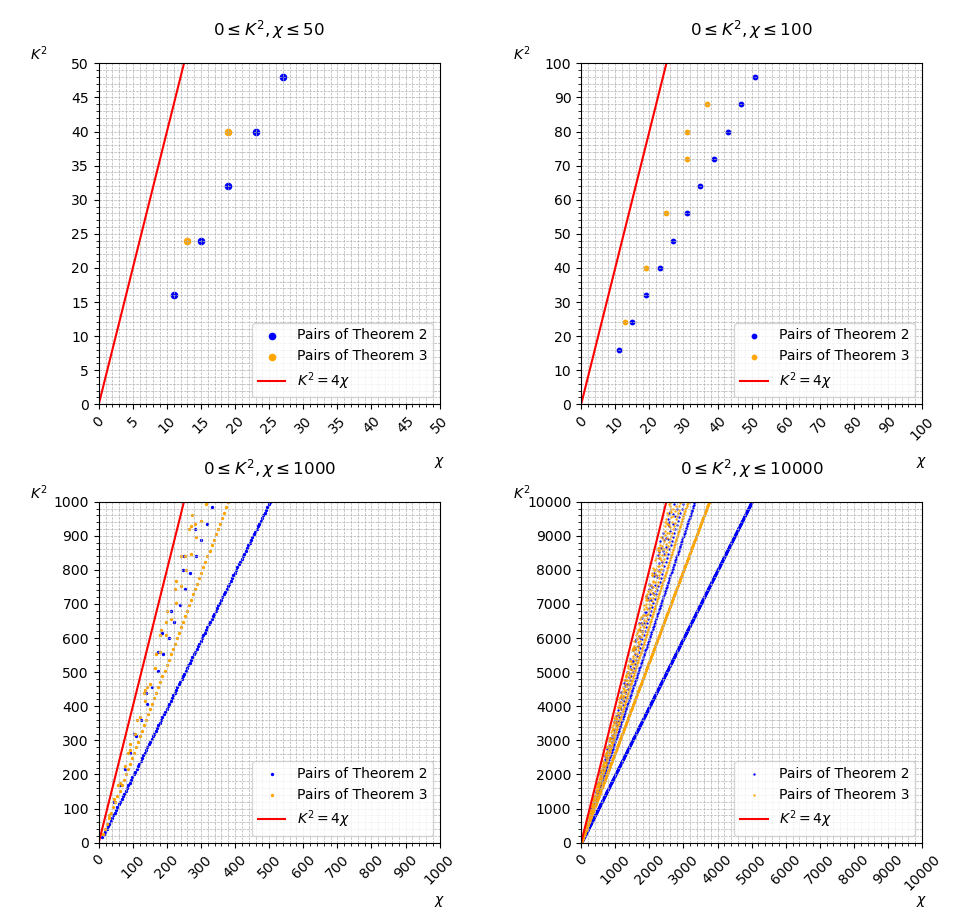}
        \caption{Multi-scale representation of the plane $(K^2,\chi)$ highlighting the pairs described in Theorem \ref{main_theorem_1}, the pairs described in Theorem \ref{main_theorem_2} and the Severi line.}
        \label{fig:linesk2chidouble}
    \end{figure}

\begin{Corollary}\label{sequences_severi} Denoting by $\mu(S)$ the slope $K^2_S/\chi(\mathcal{O}_S)$ of an algebraic surface $S$:
\begin{enumerate}
    \item[i)] For each $n\in 2\cdot \mathbb{Z}_{\geq 1}$ there exists an  unbounded sequence  $\{X^n_m\}_{m\in\mathbb{Z}_{\geq 3}}$ of surfaces of general type with maximal Picard number such that: $$\lim_{m\to\infty}{\mu(X^n_m)} = 4 - \frac{4}{n}.$$
    \item[ii)] For each $m\in \mathbb{Z}_{\geq 3}$ there exists a distinct  unbounded sequence  $\{Y^m_n\}_{n\in2\cdot \mathbb{Z}_{\geq 1}}$ of surfaces of general type with maximal Picard number such that:
    $$\lim_{n\to\infty}{\mu(Y^m_{n})} = 4.$$
\end{enumerate}
\end{Corollary}

The inequality  $K^2\geq 4\chi$ is a necessary condition for a minimal smooth complex projective surface to have maximal Albanese dimension. Although this fact was not fully proven until the 2000s by Pardini \cite{MR2125737}, $K^2\geq 4\chi$ (resp. $K^2= 4\chi$)  is known as the Severi inequality (resp. line) because Francesco Severi \cite{MR1509461} stated and  gave an incorrect proof of the aforementioned result in the 1930s.

It follows from Corollary \ref{sequences_severi}, $ii)$ that
     there exist countably many sequences of surfaces of general type with maximal Picard number that converge to the Severi line. While Persson's constructions \cite[Theorem 3]{MR661198} also give rise to a sequence $\{S_n\}$ of surfaces of general type with maximal Picard number that converges to the Severi line, our sequences are not only disjoint from $\{S_n\}$ (see Remark \ref{disjoint_constructions} below) but there are infinitely many such sequences.

     The remainder of this paper is structured as follows. Section \ref{sect:covers}  contains the basics on abelian covers
     that will be needed throughout this article. Section \ref{sect:Picard} provides an overview of key facts regarding the Picard number of algebraic surfaces that are relevant to our discussion. Finally, in Section \ref{sect:maximalPicard} we  construct families of surfaces of general type with maximal Picard number
     that will allow to prove Theorem \ref{main_theorem}, Theorem \ref{main_theorem_1} and Theorem \ref{main_theorem_2}.
     
\medskip

\textbf{Notation and conventions.} 
Throughout this paper, we assume the ground field to be
 the field of  complex numbers $\mathbb{C}$ and
all varieties to be algebraic and  projective.  We will use $\cong$ to indicate linear equivalence of line bundles or divisors, while $\equiv$ will represent congruence modulo a given integer.  The Hirzebruch surface $\mathbb{P}(\mathcal{O}_{\mathbb{P}^1}\oplus \mathcal{O}_{\mathbb{P}^1}(-e))$ will be referred to as $\mathbb{F}_e$. For simplicity, the negative section $\Delta_0$ and a general fiber $F$ of $\mathbb{F}_e$ will just mean two intersecting fibers if $e=0$. Finally, given an integer  $n$, we will denote by $\mathbb{Z}_{\geq n}$ the set of integers bigger than or equal to $n$.
The rest of the notation follows conventional usage in algebraic geometry.

\section{Abelian covers}
\label{sect:covers}

% \subsection{Structure theorem and invariants of bidouble covers}

Given a finite abelian group $G$, a finite map $f\colon X\to Y$ in conjunction with a faithful action of the group 
$G$ on $X$ such that $f$ allows $Y$ to be expressed as the quotient $X/G$, is known as a $G$-cover.
The general scenario was first explored by Pardini \cite{MR1103912}, but here we will limit our discussion to
$\mathbb{Z}_2$-covers and $\mathbb{Z}_2^2$-covers, also known as double covers and bidouble covers respectively. A key result about cyclic covers (which are $G$-covers where $G$ is a cyclic group and therefore include $\mathbb{Z}_2$-covers) of Hirzebruch surfaces will also be presented.

Comessatti \cite{comessatti1930sulle} had already studied cyclic covers of surfaces, but \cite{MR527234} is a good reference for the specific case  $G=\mathbb{Z}_2$.
The structure theorem for smooth $\mathbb{Z}_2^2$-covers was established by Catanese \cite{MR755236}.
\subsection{Double covers}
According to \cite[Section II]{MR527234} or \cite[Theorem 2.1]{MR1103912}, the building data $\{L, B\}$ suffices to define a normal $\mathbb{Z}_2$-cover $X \to Y$ of a smooth and irreducible projective variety $Y$, where:
\begin{itemize}
    \item The branch locus $B$ is a reduced and effective divisor on $Y$.
    \item $L$ is a non-trivial line bundle on $Y$ such that $L\cong 2B$.
\end{itemize}

\begin{Proposition}[{{\cite[Section II]{MR527234}}} or {{\cite[Proposition 4.2]{MR1103912}}}]\label{Z2Formulas}
Let $Y$ be a smooth surface and $f\colon X\to Y$ a smooth $\mathbb{Z}_2$-cover with building data 
	$\{L,B\}$. Then:
	\begin{equation*}
		\begin{split}
			K_X & \cong f^*(K_Y+B),\\
			K_X^2 & =2(K_Y+B)^2,\\
			p_g(X) & =p_g(Y)+h^0(K_Y+L),\\
			\chi(\mathcal{O}_X) & =2\chi(\mathcal{O}_Y)+\frac{1}{2}L(L+K_Y).
		\end{split}
	\end{equation*}
\end{Proposition}

We observe that Proposition \ref{Z2Formulas} applies when $X$ has ADE singularities (see \cite{MR2030225} for details on ADE singularities on curves and surfaces). For a broader statement, the reader is referred to Bauer and Pignatelli's work \cite[Section 2]{MR4278662}.

\subsection{Bidouble covers}
According to \cite[Section 2]{MR755236} or \cite[Theorem 2.1]{MR1103912}, the building data $\{L_i, B_j\}_{i,j\in\{1,2,3\}}$ suffices to define a normal $\mathbb{Z}_2^2$-cover $X \to Y$ of a smooth and irreducible projective variety $Y$, where:
\begin{itemize}
	\item $B_1, B_2, B_3$ are effective divisors on $Y$ such that the branch locus $B=B_1+B_2+B_3$ is reduced.
	\item $L_1, L_2, L_3$ are non-trivial line bundles  on $Y$ satisfying $2L_1 \cong B_2 + B_3$, $2L_2 \cong B_1 + B_3$ and $L_3 \cong L_1 + L_2 - B_3$.
\end{itemize}

\begin{Proposition}[{{\cite[Section 2]{MR755236}}} or {{\cite[Proposition 4.2]{MR1103912}}}]\label{Z22Formulas}
	Let $Y$ be a smooth surface and $f\colon X\to Y$ a smooth $\mathbb{Z}_2^2$-cover with building data 
	$\{L_i,B_j\}_{i,j}$. Then:
	\begin{equation*}
		\begin{split}
			2K_X & \cong f^*(2K_Y+B_1+B_2+B_3),\\
			K_X^2 & =(2K_Y+B_1+B_2+B_3)^2,\\
			p_g(X) & =p_g(Y)+\sum_{i=1}^3h^0(K_Y+L_i),\\
			\chi(\mathcal{O}_X) & =4\chi(\mathcal{O}_Y)+\frac{1}{2}\sum_{i=1}^3L_i(L_i+K_Y).
		\end{split}
	\end{equation*}
\end{Proposition}

We observe that Proposition \ref{Z22Formulas} also applies when $X$ has ADE singularities. Again, the reader is referred to \cite[Section 2]{MR4278662}.

% \subsection{Some singularities arising from bidouble covers}\label{SubSingZ22}

\begin{Corollary}[{\cite[Corollary 1]{MR4761778}}]\label{Z22Sing}
	Let $f\colon X\to Y$ be a $\mathbb{Z}_2^2$-cover of a smooth surface $Y$ with  building data $\{L_i,B_j\}_{i,j}$ and let $Q=f(P)$ be an intersection point of $B_1$ and $B_2$ that is not contained in $B_3$.
	\begin{itemize}
		\item[i)] If both $B_1$ and $B_2$ are smooth at $Q$ but they intersect in such a way that $B_1+B_2$ has a singularity of type $A_{2n+1},n\geq 1$ at $Q$, then  $X$ has a singularity of type $A_n$ at $P$.
%		\item[ii)] If $B_1$ has a singularity of type $A_1$ at $Q$, $B_2$ is smooth at $Q$ and they intersect in such a way that $B_1+B_2$ has a singularity of type $D_{2n+4}, n\geq 1$ at $Q$, then  $X$ has a singularity of type $D_{n+3}$ at $P$.
		\item[ii)] If $B_1$ has a singularity of type $A_{n}, n\geq 1$ at $Q$, $B_2$ is smooth at $Q$ and they intersect in such a way that $B_1+B_2$ has a singularity of type $D_{n+3}$ at $Q$, then  $X$ has a singularity of type $A_{2n+1}$ at $P$.
	\end{itemize}
\end{Corollary}

\begin{Remark}\label{DoubleSing}
It is worth noting that:
\begin{enumerate}
    \item[i)]	 If $n=0$ in Corollary \ref{Z22Sing}, i) then $X$ is smooth at $P$.
    \item[ii)]	 An ADE singularity of $B_i,i\in\{1,2,3\}$ that is disjoint from the two other divisors of the branch locus induces two singularities of the same type on $X$.
\end{enumerate}
\end{Remark}

For more information about the singularities arising from bidouble covers  the reader is addressed to \cite{MR879190} or  \cite{MR1103912}.

\subsection{Cyclic covers of Hirzebruch surfaces}

The following result is a restatement of \cite[Lemma 4.1]{MR661198} that will be essential to prove Theorem \ref{main_theorem_1} and Theorem \ref{main_theorem_2}.

\begin{Lemma}[{\cite[Lemma 4.1]{MR661198}} or {\cite[Lemma 2]{MR4761778}}]\label{PerssonsTrick}
	Given a Hirzebruch surface $\mathbb{F}_e$ with two disjoint fibers $F_1,F_2$ there is for
	each integer $d$ a unique $\mathbb{Z}_d$-cover $\pi_d\colon \mathbb{F}_{de} \to \mathbb{F}_e$ branched at $F_1 $ and $ F_2$. Furthermore:
	\begin{itemize}
		\item[a)] If $C\cong \mathcal{O}_{\mathbb{F}_e}(a\Delta_0+b F)$ is an effective divisor not having $F_1$ and $F_2$ as components then $C_d:=\pi_d^*C\cong \mathcal{O}_{\mathbb{F}_{de}}(a\Delta_0+dbF)$.
		\item[b)] If $C$ has at most ADE singularities, then $C_d$ has at most ADE singularities if and only if the only singularities of
		$C$ lying on $F_1$ or $F_2$ are of type $A_n$ and they are transversal to these fibers.
		\item[c)] Given $n\in2\cdot \mathbb{Z}_{\geq 1}$ a singularity of type $A_{n-1}$
  %$A_{2n-1}$ 
  of $C$ on $F_i,i\in\{1,2\}$ that is transversal to this fiber gives rise to a singularity of type %$A_{2dn-1}$
  $A_{dn-1}$
  of $C_d$. 
		\item[d)] A singularity of $C$ not lying on $F_1\cup F_2$ gives rise to $d$ singularities of the same type on $C_d$.
	\end{itemize}
\end{Lemma}

\begin{Remark}
The authors suspect that Lemma \ref{PerssonsTrick}, c) remains valid when $n \in \mathbb{Z}_{\geq 2}$ is odd, despite this case not being addressed in Persson's original statement \cite[Lemma 4.1, c)]{MR661198}.
 To support this, they explicitly wrote $\pi_d$ using the description of Hirzebruch surfaces provided in \cite[Chapter 2]{MR1442522}. Then, they studied whether the pullback of a singularity of type $A_{n-1}$ of $C$ on $F_i,i\in\{1,2\}$ that is transversal to this fiber gives rise to a singularity of type %$A_{2dn-1}$
  $A_{dn-1}$
  of $C_d$. If confirmed, this claim could broaden the scope of Theorem \ref{main_theorem_1} and Theorem \ref{main_theorem_2}. However, the authors have chosen not to include it as a precautionary measure.
\end{Remark}

\section{On the Picard Number of an algebraic surface}\label{sect:Picard}

Determining the Picard number of a smooth projective surface can be challenging in general. However, there is a relatively straightforward method to demonstrate that a smooth projective surface $X$ has maximal Picard number. 

Indeed, on the one hand, the number of algebraically independent divisor classes in $NS(X)$ that one is able to identify is a lower bound on $\rho(X)$.
    On the other hand,  $h^{1,1}(X)=10\chi(\mathcal{O}_X)-K^2_X-2q(X)$
is an upper bound on $\rho(X)$.

Hence, if we are able to find $h^{1,1}(X)$ divisors on  $X$ whose intersection matrix has rank $h^{1,1}(X)$,  then $X$ has maximal Picard number. See \cite[Section 2]{MR2855811} for other approaches.

From the discussion above one can easily infer the following:

\begin{Lemma}[{\cite[Lemma 1]{MR4761778}}]   \label{PicardSurjectiveMorphism}
	Let $X$ be a canonical model with minimal resolution $\hat{X}\to X$ and whose singular set consists of:
	\begin{itemize}
		\item $\alpha_i$ singularities of type $A_i, i\in\mathbb{Z}_{\geq 1}$,
		\item $\beta_j$ singularities of type $D_j, j\in\mathbb{Z}_{\geq 4}$,
		\item $\gamma_k$ singularities of type $E_k, k\in\{6,7,8\}$.
	\end{itemize}
	Let us suppose that $X$ admits a surjective morphism $\pi:X\to Y$ onto a smooth and projective surface $Y$ and there exist numerically independent divisors $C_1,\ldots, C_n$ on $Y$.
	Then:
	\begin{equation*}
		\rho(\hat{X})\geq \sum_{i} i\cdot\alpha_i + 
		\sum_{j} j\cdot\beta_j+\sum_{k} k\cdot\gamma_k+n.
	\end{equation*}
\end{Lemma}

\begin{Remark}
	 Lemma \ref{PicardSurjectiveMorphism} will just be applied in the following scenarios:
	\begin{itemize}
		\item When $\pi\colon X\to \mathbb{P}^2$   is a $\mathbb{Z}_2^2$-cover of the projective plane $\mathbb{P}^2$, $n=1$ and $C_1$ is a line.
		\item When $\pi\colon X\to \mathbb{F}_e$ is either a $\mathbb{Z}_2$-cover or a $\mathbb{Z}_2^2$-cover of the Hirzebruch surface $\mathbb{F}_e$ for some integer $e \geq 0$, $n=2$ and the divisors $C_1,C_2$ are the negative section and a fiber.
	\end{itemize}
\end{Remark} 

\section{Surfaces with maximal Picard number}\label{sect:maximalPicard}
This section is devoted to construct families of surfaces of general type with maximal Picard number from which Theorem \ref{main_theorem}, Theorem \ref{main_theorem_1} and Theorem \ref{main_theorem_2} will follow. Given that all these surfaces are obtained as abelian covers whose branch locus derives from the curve
 \begin{equation*}
C=[(X_0^n+X_1^n+X_2^n)^2-4\cdot ((X_0X_1)^n+(X_0X_2)^n+(X_1X_2)^n)=0]
 \end{equation*}
 on $\mathbb{P}^2=\text{Proj}(\mathbb{C}[X_0,X_1,X_2])$, we will begin by establishing the common setup and notation.

 Denoting $l_1 = [X_0 = 0]$, $l_2 = [X_1 = 0]$ and $l_3 = [X_2 = 0]$, the curve $C$ has $n$ singularities of type $A_{n-1}$ on $l_1$, $n$ singularities of type $A_{n-1}$ on $l_2$ and $n$ singularities of type $A_{n-1}$ on $l_3$ by \cite[Lemma 7.5]{MR661198}.\\

\begin{figure}[H]
\begin{center}
\tikzset{every picture/.style={line width=0.75pt}} %set default line width to 0.75pt        

\begin{tikzpicture}[x=0.75pt,y=0.75pt,yscale=-1,xscale=1]
%uncomment if require: \path (0,310); %set diagram left start at 0, and has height of 310

%Straight Lines [id:da5068165925218464] 
\draw [color={rgb, 255:red, 0; green, 0; blue, 0 }  ,draw opacity=1 ]   (307,31.6) -- (133,222.6) ;
%Straight Lines [id:da048630773549271034] 
\draw [color={rgb, 255:red, 0; green, 0; blue, 0 }  ,draw opacity=1 ]   (128,204.6) -- (486,202.6) ;
%Straight Lines [id:da6116215644023424] 
\draw [color={rgb, 255:red, 0; green, 0; blue, 0 }  ,draw opacity=1 ]   (268,35.6) -- (478,225.6) ;
%Shape: Circle [id:dp035931760747846875] 
\draw  [fill={rgb, 255:red, 74; green, 144; blue, 226 }  ,fill opacity=1 ] (284,54) .. controls (284,52.34) and (285.34,51) .. (287,51) .. controls (288.66,51) and (290,52.34) .. (290,54) .. controls (290,55.66) and (288.66,57) .. (287,57) .. controls (285.34,57) and (284,55.66) .. (284,54) -- cycle ;
%Shape: Circle [id:dp2899837187610794] 
\draw  [color={rgb, 255:red, 0; green, 0; blue, 0 }  ,draw opacity=1 ][fill={rgb, 255:red, 0; green, 0; blue, 0 }  ,fill opacity=1 ] (450,203) .. controls (450,201.34) and (451.34,200) .. (453,200) .. controls (454.66,200) and (456,201.34) .. (456,203) .. controls (456,204.66) and (454.66,206) .. (453,206) .. controls (451.34,206) and (450,204.66) .. (450,203) -- cycle ;
%Shape: Circle [id:dp5697779154366291] 
\draw  [color={rgb, 255:red, 0; green, 0; blue, 0 }  ,draw opacity=1 ][fill={rgb, 255:red, 0; green, 0; blue, 0 }  ,fill opacity=1 ] (147,204) .. controls (147,202.34) and (148.34,201) .. (150,201) .. controls (151.66,201) and (153,202.34) .. (153,204) .. controls (153,205.66) and (151.66,207) .. (150,207) .. controls (148.34,207) and (147,205.66) .. (147,204) -- cycle ;
%Shape: Spring [id:dp16362933448956474] 
\draw   (163.68,193.31) .. controls (161.26,187.7) and (161.54,179.14) .. (170.32,169.55) .. controls (187.88,150.38) and (204.32,165.44) .. (200.27,169.86) .. controls (196.21,174.29) and (179.78,159.23) .. (197.34,140.06) .. controls (214.9,120.89) and (231.34,135.94) .. (227.29,140.37) .. controls (223.23,144.79) and (206.79,129.73) .. (224.36,110.56) .. controls (241.92,91.39) and (258.36,106.45) .. (254.3,110.87) .. controls (250.25,115.3) and (233.81,100.24) .. (251.37,81.07) .. controls (268.94,61.89) and (285.38,76.95) .. (281.32,81.38) .. controls (278.1,84.89) and (267.05,76.1) .. (271.48,62.71) ;
%Shape: Spring [id:dp4140045345664063] 
\draw   (312.43,80.46) .. controls (318.02,77.98) and (326.58,78.18) .. (336.25,86.86) .. controls (355.6,104.23) and (340.7,120.82) .. (336.24,116.81) .. controls (331.77,112.8) and (346.67,96.22) .. (366.01,113.59) .. controls (385.36,130.96) and (370.46,147.55) .. (366,143.54) .. controls (361.54,139.53) and (376.43,122.94) .. (395.78,140.31) .. controls (415.12,157.68) and (400.23,174.27) .. (395.76,170.26) .. controls (391.3,166.25) and (406.19,149.66) .. (425.54,167.03) .. controls (444.88,184.4) and (429.99,200.99) .. (425.53,196.98) .. controls (421.78,193.62) and (431.65,181.42) .. (446.33,187.8) ;
%Shape: Spring [id:dp5990623309765049] 
\draw   (217.88,208.42) .. controls (220.38,202.84) and (226.87,197.26) .. (239.87,197.24) .. controls (265.87,197.21) and (265.9,219.5) .. (259.9,219.51) .. controls (253.9,219.52) and (253.87,197.23) .. (279.87,197.19) .. controls (305.87,197.16) and (305.9,219.46) .. (299.9,219.46) .. controls (293.9,219.47) and (293.87,197.18) .. (319.87,197.15) .. controls (345.87,197.11) and (345.9,219.41) .. (339.9,219.41) .. controls (333.9,219.42) and (333.87,197.13) .. (359.87,197.1) .. controls (385.87,197.06) and (385.9,219.36) .. (379.9,219.37) .. controls (374.45,219.37) and (373.93,200.97) .. (393.37,197.58) ;

% Text Node
\draw (404,124) node [anchor=north west][inner sep=0.75pt]  [color={rgb, 255:red, 0; green, 0; blue, 0 }  ,opacity=1 ] [align=left] {$C$};
% Text Node
\draw (294,46) node [anchor=north west][inner sep=0.75pt]  [color={rgb, 255:red, 74; green, 144; blue, 226 }  ,opacity=1 ] [align=left] {$P$};
% Text Node
\draw (249,180) node [anchor=north west][inner sep=0.75pt]  [color={rgb, 255:red, 0; green, 0; blue, 0 }  ,opacity=1 ] [align=left] {$A_{n-1}$};
% Text Node
\draw (239,122) node [anchor=north west][inner sep=0.75pt]  [color={rgb, 255:red, 0; green, 0; blue, 0 }  ,opacity=1 ] [align=left] {$l_1$};
% Text Node
\draw (353,78) node [anchor=north west][inner sep=0.75pt]  [color={rgb, 255:red, 0; green, 0; blue, 0 }  ,opacity=1 ] [align=left] {$A_{n-1}$};
% Text Node
\draw (200,82) node [anchor=north west][inner sep=0.75pt]  [color={rgb, 255:red, 0; green, 0; blue, 0 }  ,opacity=1 ] [align=left] {$A_{n-1}$};
% Text Node
\draw (194,219) node [anchor=north west][inner sep=0.75pt]  [color={rgb, 255:red, 0; green, 0; blue, 0 }  ,opacity=1 ] [align=left] {$l_3$};
% Text Node
\draw (337,122) node [anchor=north west][inner sep=0.75pt]  [color={rgb, 255:red, 0; green, 0; blue, 0 }  ,opacity=1 ] [align=left] {$l_2$};
\end{tikzpicture}
\caption{Schematic depiction of curve $C$.}
\end{center}
\end{figure}
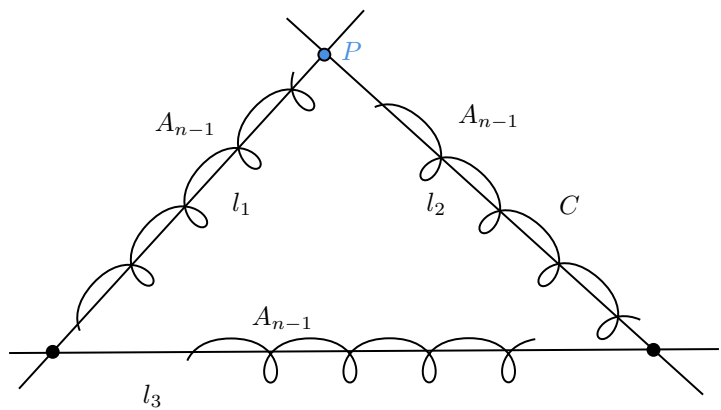

\noindent
Let $ \pi\colon \mathbb{F}_1 \to \mathbb{P}^2 $ be the blow-up of $ \mathbb{P}^2 $ at $ P $ where  $P$ is the intersection of $l_1$ and $l_2$. We denote by:
	\begin{center}
		\begin{tabular}{l l}
			$ \overline{l}_3, \overline{C}$: & the pull-back via $\pi$ of  $ l_3, C$  respectively,\\
			$ l'_1, l'_2 $:&the strict transforms via $\pi$ of $ l_1 $, $ l_2 $ respectively.
		\end{tabular}	
	\end{center}
Note that 
    $l'_1\cong l'_2 \cong \mathcal{O}_{\mathbb{F}_1}(F)$, $\overline{l}_3\cong \mathcal{O}_{\mathbb{F}_1}(\Delta_0+F)$, and $\overline{C}\cong \mathcal{O}_{\mathbb{F}_1}(2n\Delta_0+2nF)$.\\

Let $\psi\colon \mathbb{F}_{m} \to \mathbb{F}_1  $ be the $\mathbb{Z}_{m}$-cover branched at $ l'_1 + l'_2 $ (see Lemma \ref{PerssonsTrick}).\\

Firstly, since $\overline{l}_3\cong \mathcal{O}_{\mathbb{F}_1}(\Delta_0+F)$  and $\overline{C}\cong \mathcal{O}_{\mathbb{F}_1}(2n\Delta_0+2nF)$ are effective divisors not having $l'_1 $ and $ l'_2$ as components,  then
    $\tilde{l}_3 :=\psi^{*}\left( \overline{l}_3\right) \in |\mathcal{O}_{\mathbb{F}_{m}}(\Delta_0 + mF)|$ and $\tilde{C} :=\psi^{*}\left( \overline{C}\right) \in |\mathcal{O}_{\mathbb{F}_{m}}(2n\Delta_0 + 2nmF)|$
   by Lemma \ref{PerssonsTrick}, where $\psi^*$ is the pull-back of $\psi$.\\

Secondly, being $\psi$   branched at $ l'_1 + l'_2 $, we obtain that    
    \begin{equation*}
		\begin{split}			
			\psi^{*}\left( l'_1\right) & =m\tilde{l}_1,\\
			\psi^{*}\left( l'_2\right) & =m\tilde{l}_2,
		\end{split}
	\end{equation*}
    
    for some fibers $\tilde{l}_1, \tilde{l}_2 \in |\mathcal{O}_{\mathbb{F}_{m}}(F)|$.\\

\begin{proof}[Proof of Theorem \ref{main_theorem}]
We consider the following divisors on $ \mathbb{P}^{2} $:
	\begin{equation*}
		\begin{split}
			& B_{1} := l_1,\\
			& B_{2}: =l_2, \\ 
			& B_{3} : = l_3 + C.
		\end{split}
	\end{equation*}
	
	Note that $B_1, B_2\in |\mathcal{O}_{\mathbb{P}^{2}}(1)|$ and $ B_3 \in |\mathcal{O}_{\mathbb{P}^{2}}(2n+1))| $.\\
 
We also consider the following line bundles on $ \mathbb{P}^2 $:
	\begin{equation*}
		\begin{split}
			& L_1:= \mathcal{O}_{\mathbb{P}^{2}}(n+1), \\
			& L_2:=\mathcal{O}_{\mathbb{P}^{2}}(n+1),\\
			& L_3:= \mathcal{O}_{\mathbb{P}^{2}}(1).
		\end{split}
	\end{equation*}

The building data $\{L_i,B_j\}_{i,j}$ defines a $ \mathbb{Z}^2_2 $-cover $ \varphi\colon X \to \mathbb{P}^{2}$. On the one hand, taking into account \cite[Lemma 7.5]{MR661198}, we get:
	\begin{itemize}
		\item The singular set of $B_3$ restricted to $B_3\setminus(B_1\cup B_2)$ consists of $n$ singularities of type $D_{n+2}$.
		\item The singular set of $B_1+B_3$ restricted to $ B_1\cap B_3 $ consists of $n$ singularities of type $D_{n +2}$.
		\item The singular set of $B_2+B_3$ restricted to $ B_2\cap B_3 $ consists of $n$ singularities of type $D_{n +2}$.
	\end{itemize}
Thus, the singular set of $X$ consists of: 
	\begin{itemize}
		\item $2n$ singularities of type $D_{n+2}$ by Remark \ref{DoubleSing}.
		\item $ n $ singularities of type $ A_{2n-1} $ by Corollary \ref{Z22Sing},$ii)$.
		\item $ n $ singularities of type $ A_{2n-1} $ by Corollary \ref{Z22Sing},$ii)$.
	\end{itemize}
 
On the other hand, by Proposition \ref{Z22Formulas} we have that $K_X$ is ample because $2K_{X}$ is the pull-back via $\varphi$ of the ample divisor $$\mathcal{O}_{\mathbb{P}^{2}}\left(2n-3\right).$$
Moreover:
	\begin{equation*}
		\begin{split}
			K_{X}^2 & = 4n^2 - 12n +9,\\
			\chi\left( \mathcal{O}_{X}\right) &=n^2 -n+1,\\
			p_g\left( X\right) & =n^2 -n,\\
			q\left( X\right)  & = 0,\\
			h^{1,1}\left( X\right) & = 6n^2 + 2n +1.
		\end{split}
	\end{equation*}	
	We conclude that $X\in\mathfrak{M}_{K^2,\chi}$ is a canonical model whose minimal resolution has maximal Picard number $  h^{1,1}\left( X\right) = 6n^2 + 2n +1$ by Lemma \ref{PicardSurjectiveMorphism}.\\
\end{proof}

\begin{proof}[Proof of Theorem \ref{main_theorem_1}]

We consider the following divisors on $ \mathbb{F}_{m} $:
	\begin{equation*}
		\begin{split}
			& B_{1} := 0,\\
			& B_{2}: =\tilde{l}_1+\tilde{l}_2 \in|\mathcal{O}_{\mathbb{F}_{m}}(2F)|,\\ 
			& B_{3} : = \Delta_0 + \tilde{l}_3 + \tilde{C} \in |\mathcal{O}_{\mathbb{F}_{m}}(\left(2n+2\right)\Delta_0 + \left( 2n+1\right)mF)|,
		\end{split}
	\end{equation*}
 if $m$ is even and:
 	\begin{equation*}
		\begin{split}
			& B_{1} := \tilde{l}_1\in|\mathcal{O}_{\mathbb{F}_{m}}(F)|,\\
			& B_{2}: =\tilde{l}_2\in|\mathcal{O}_{\mathbb{F}_{m}}(F)| ,\\ 
			& B_{3} : = \Delta_0 + \tilde{l}_3 + \tilde{C} \in |\mathcal{O}_{\mathbb{F}_{m}}(\left(2n+2\right)\Delta_0 + \left( 2n+1\right)mF)|,
		\end{split}
	\end{equation*}
 if $m$ is odd.\\

 We also consider the following line bundles on $ \mathbb{F}_{m} $:
	\begin{equation*}
		\begin{split}
			& L_1:= \mathcal{O}_{\mathbb{F}_{m}}\left(\left( n+1\right)\Delta_0+\left(mn+\frac{1}{2}{\color{blue}n}+1\right)F\right), \\
			& L_2:=\mathcal{O}_{\mathbb{F}_{m}}\left(\left( n+1\right)\Delta_0+\left(mn+\frac{1}{2}{\color{blue}n}\right)F\right),\\
			& L_3:= \mathcal{O}_{\mathbb{F}_{m}}\left(F\right),
		\end{split}
	\end{equation*}
 if $m$ is even and:
 	\begin{equation*}
		\begin{split}
			& L_1:= \mathcal{O}_{\mathbb{F}_{m}}\left(\left( n+1\right)\Delta_0+\left(mn+\frac{1}{2}({\color{blue}n}-1)+1\right)F\right), \\
			& L_2:=\mathcal{O}_{\mathbb{F}_{m}}\left(\left( n+1\right)\Delta_0+\left(mn+\frac{1}{2}({\color{blue}n}-1)+1\right)F\right),\\
			& L_3:= \mathcal{O}_{\mathbb{F}_{m}}\left(F\right),
		\end{split}
	\end{equation*}
 if $m$ is odd.\\

The building data $\{L_i,B_j\}_{i,j}$ defines a $ \mathbb{Z}^2_2 $-cover $ \varphi\colon X \to \mathbb{F}_{m}$. On the one hand, taking into account \cite[Lemma 7.5]{MR661198}, we get:
	\begin{itemize}
		\item The singular set of $B_3$ restricted to $B_3\setminus(B_1\cup B_2)$ consists of $mn$ singularities of type $D_{n+2}$.
		\item The singular set of $B_1+B_2+B_3$ restricted to $ (B_1\cap B_3) \cup (B_2\cap B_3)$ consists of $2n$ singularities of type $D_{mn +2}$.
	\end{itemize}
Thus, the singular set of $X$ consists of: 
	\begin{itemize}
		\item $2mn$ singularities of type $D_{n+2}$ by Remark \ref{DoubleSing}.
		\item $2 n $ singularities of type $ A_{2mn-1} $ by Corollary \ref{Z22Sing},$ii)$.
	\end{itemize}
 
On the other hand, by Proposition \ref{Z22Formulas} we have that $K_X$ is ample because $2K_{X}$ is the pull-back via $\varphi$ of the ample divisor $$\mathcal{O}_{\mathbb{F}_{m}}\left(\left( 2n-2\right)\Delta_0 + \left(2mn-m-2\right)F\right).$$
Moreover:
	\begin{equation*}
		\begin{split}
			K_{X}^2 & = \left(\left( 2n-2\right)\Delta_0 + \left(2mn-m-2\right)F\right)^2=4mn^2 - 4\left( m+2\right)n +8,\\
			\chi\left( \mathcal{O}_{X}\right) &=mn^2 -n+1,\\
			p_g\left( X\right) & =mn^2 -n,\\
			q\left( X\right)  & = 0,\\
			h^{1,1}\left( X\right) & = 6mn^2 + \left( 4m-2\right)n +2.
		\end{split}
	\end{equation*}	
	We conclude that $X\in\mathfrak{M}_{K^2,\chi}$ is a canonical model whose minimal resolution has maximal Picard number $  h^{1,1}\left( X\right) = 6mn^2 + \left( 4m-2\right)n +2$ by Lemma \ref{PicardSurjectiveMorphism}.
\end{proof}

\begin{proof}[Proof of Theorem \ref{main_theorem_2}]

We consider following the divisor on $ \mathbb{F}_{m} $:
$$B = \tilde{C} + \tilde{l}_1 +\tilde{l}_2 \in |\mathcal{O}_{\mathbb{F}_{m}}(2n\Delta_0 + \left( 2mn+2\right)F)|$$

The divisor $B$ defines a $ \mathbb{Z}_2 $-cover $ \varphi\colon X \to \mathbb{F}_{m}$. On the one hand, taking into account \cite[Lemma 7.5]{MR661198}, we get that the singular set of $B$ consists of $2n$ singularities of type $D_{mn+2}$ and $mn$ singularities of type $A_{n-1}$.\\
	
Thus, the singular set of $X$ consists of $2n$ singularities of type $D_{mn+2}$ and $mn$ singularities of type $A_{n-1}$.\\

On the other hand, we have that $K_X$ is ample because $K_{X}$ is the pull-back via $\varphi$ of the ample divisor $$\mathcal{O}_{\mathbb{F}_{m}}\left(\left( n-2\right)\Delta_0 + \left(mn-m-1\right)F\right).$$
Moreover:
	\begin{equation*}
		\begin{split}
			K_{X}^2 & = 2\left(\left( n-2\right)\Delta_0 + \left(mn-m-1\right)F\right)^2=2mn^2 - 4\left( m+1\right)n +8,\\
			\chi\left( \mathcal{O}_{X}\right) &=\frac{1}{2}mn(n-1)+1,\\
			p_g\left( X\right) & =\frac{1}{2}mn(n-1),\\
			q\left( X\right)  & = 0,\\
			h^{1,1}\left( X\right) & = 3mn^2 + \left( 4-m\right)n +2.
		\end{split}
	\end{equation*}	
	We conclude that $X\in\mathfrak{M}_{K^2,\chi}$ is a canonical model whose minimal resolution has maximal Picard number $  h^{1,1}\left( X\right) = 3mn^2 + \left( 4-m\right)n +2$ by Lemma \ref{PicardSurjectiveMorphism}.
\end{proof}

\begin{proof}[Proof of Corollary \ref{sequences_severi}]
Given $m\in\mathbb{Z}_{\geq 3}$ and $n\in 2\cdot \mathbb{Z}_{\geq 1}$, let us consider  the surface $X_{m,n}$ described in Theorem \ref{main_theorem_1}. 
    Then:
    \begin{equation}\label{slope_limit}
\mu(X_{m,n})=\frac{K^2_{X_{m,n}}}{\chi(\mathcal{O}_{X_{m,n}})}=4+\frac{4(1-n-mn)}{1-n+mn^2}.
%\xrightarrow[n\to\infty]{}4
\end{equation}
Defining $X_m^n=Y_n^m=X_{m,n}$, the result easily follows from equation (\ref{slope_limit}).
\end{proof}

\begin{Remark}\label{disjoint_constructions}
Let us consider the following subsets of $\mathbb{Z}\times \mathbb{Z}$:
\begin{equation*}
    \begin{split}
        A_1 & :=\left\{\left(4n^2-12n+9,n^2-n+1\right):n\in\mathbb{Z}_{\geq 2}\right\},\\
        A_2 & :=\left\{\left(4mn^2-4(m+2)n+8,mn^2-n+1\right):m\in\mathbb{Z}_{\geq 3},n\in2\cdot \mathbb{Z}_{\geq 1}\right\},\\
        A_3 & :=\left\{\left(2mn^2-4(m+1)n+8 ,\frac{1}{2}mn(n -1)+1\right):m\in\mathbb{Z}_{\geq 2}, n\in2\cdot \mathbb{Z}_{\geq 2}\right\},\\
        B & := \left\{\left(2(n-3)^2, \frac{1}{2}(n-1)(n-2)+1\right):n\in\mathbb{Z}_{\geq 4}\right\}.
    \end{split}
\end{equation*}
These sets contain the invariants of the surfaces of Theorem \ref{main_theorem}, Theorem \ref{main_theorem_1}, Theorem \ref{main_theorem_2} and \cite[Theorem 3]{MR661198} respectively. We claim that:
\begin{enumerate}
    \item[i)] $A_1\cap B=A_2\cap B=A_3\cap B=A_1\cap A_2=A_1\cap A_3=\varnothing$.
    \item[ii)] %$A_2 \cap A_3\neq \varnothing$ but 
    Both $A_2\setminus A_3$ and $A_3\setminus A_2$ contain an infinite amount of pairs.
    \item[iii)]  $A_2 \cap A_3$  contains an infinite amount of pairs.
\end{enumerate}
\end{Remark}
The equalities $A_1\cap B=A_2\cap B=A_3\cap B=\varnothing$ follow noticing that $\frac{1}{2}K^2$ is a perfect square if $(K^2,\chi)\in B$, which is not the case for  pairs $(K^2,\chi)\in A_1\cup A_2\cup A_3$.\\
The equalities $A_1\cap A_2=A_1\cap A_3=\varnothing$ follow noticing that $K^2$ is odd for every pair $(K^2,\chi)\in A_1$ whereas $K^2$ is even for every pair $(K^2,\chi)\in A_2\cup A_3$.\\
%The set $A_2\cap A_3$ is not empty because at least $(32,19)\in A_2\cap A_3$.\\
The fact that $A_3\setminus A_2$ contains an infinite amount of pairs follows noticing that $\chi$ is odd for every pair  $(K^2,\chi)\in A_2$ but it is easy to construct an infinite sequence of pairs $(K^2,\chi)\in A_3$ such that $\chi$ is even.\\
In order to prove that $A_2\setminus A_3$ contains an infinite amount of pairs it suffices to notice that the intersection of $A_2$ with the Noether line is $$\{(8m-8,4m-1):m\in\mathbb{Z}_{\geq 3}\}$$   and the intersection of $A_3$ with the Noether line is empty.\\
To prove that $A_2 \cap A_3$  contains an infinite amount of pairs it suffices to show that $A_2 \cap A_3$  contains the  set
    $$T: =\left\{\left(2t(t-1)(t-4) +8,\frac{1}{2}t(t-1)(t-3) +1\right): t\in2\cdot \mathbb{Z}_{\geq 3}\right\}.$$
\noindent
Now, setting $m = \frac{t}{2}-1$ and $n = t-1$, we get $T \subseteq A_2 $, and setting $m = t-3$ and $n = t$, we get $T \subseteq A_3 $.

\iffalse
Persson's constructions \cite[Theorem 3]{MR661198}  give rise to a sequence of surfaces of general type with maximal Picard number that converges to the Severi line. Nevertheless, these examples are disjoint from those of Theorem \ref{main_theorem_1}. Indeed, it suffices to show that  $$A:=\{(K^2_{X_{m,n}}, \chi(\mathcal{O}_{X_{m,n}})):m, n\geq 2\}$$  is disjoint from $$B:=\left\{\left(2(n-3)^2, \frac{(n-1)(n-2)}{2}+1\right):n\geq 4\right\}$$ as these sets contain the invariants of the surfaces of Theorem \ref{main_theorem_1} and \cite[Theorem 3]{MR661198} respectively. Note that for $(K^2_{X_{k,n}}, \chi(\mathcal{O}_{X_{k,n}}))$ to belong to $B$ we need $$\frac{K^2_{X_{k,n}}}{2}=2(2k+1)n^2-2(2k+3)n+4$$ to be a perfect square. Now, this will happen if and only if the discriminant $4(2k+3)^2-32(2k+1)$ of $\frac{K^2_{X_{k,n}}}{2}$, which is a degree $2$ polynomial in $n$, is $0$. Since this is only the case if $k=\frac{1}{2}$ and we are assuming $k$ is an integer, we conclude that $A$ and $B$ are indeed disjoint.
\fi

\begin{Remark}
    Note that, while some of our examples lie in the region $K^2\leq \frac{5}{2}\chi-11$ covered by \cite[Theorem 1]{MR4761778}, we have constructed infinitely many surfaces of general type with maximal Picard number that do not. This claim follows from Corollary \ref{sequences_severi} taking into account that $K^2/\chi\leq \frac{5}{2}$ if $K^2\leq \frac{5}{2}\chi-11$.
\end{Remark}

\begin{Remark}
    As Figure \ref{fig:linesk2chidouble} suggests, the examples of Theorem \ref{main_theorem_1}  and Theorem \ref{main_theorem_2} are arranged in lines. % Furthermore, the examples of Theorem \ref{main_theorem_1} lie on lines of the plane $(K^2, \chi)$ where examples of Theorem \ref{main_theorem_2} can be found. \\
    
   On the one hand, given $n\in2\cdot \mathbb{Z}_{\geq 1}$, it is easy to see that the surface $X_{m,n}$ of Theorem \ref{main_theorem_1} lies on the line
    $$K^2=4\frac{n-1}{n}\chi-4\frac{(n+1)(n-1)}{n}$$
    for every integer $m\geq 3$.
    
    On the other hand, given an integer 
    $n\in2\cdot \mathbb{Z}_{\geq 2}$, it is easy to see that the surface $X_{m,n}$ of Theorem \ref{main_theorem_2} lies on the line
    $$K^2=4\frac{n-2}{n-1}\chi-4\frac{n(n-2)}{n-1}$$
    for every integer $m\geq 2$.
\end{Remark}
\iffalse
\begin{Remark}
    Figure \ref{fig:linesk2chidouble} suggests that density results such as \cite[Remark 8]{MR4761778} are not expected to be derived from Theorem \ref{main_theorem_1} or Theorem \ref{main_theorem_2}. More specifically,
    using the notation of Corollary \ref{sequences_severi} and its proof, the closure of the set $\{\mu(S_n):n\geq 2\}$ contains no open interval.
\end{Remark}
\fi

 \begin{Acknowledgments}
    The authors are grateful to Professor Margarida Mendes Lopes and Professor Jungkai  Alfred Chen for their valuable feedback on this article. This research is funded by Vietnam Ministry of Education and Training (through National Key Program for the Development of Mathematics in the 2021–2030 period) under grant number B2025-CTT-04. 
 \end{Acknowledgments}

\bibliographystyle{plain} 
\bibliography{Vicente-Bin}

\begin{paracol}{2}
\BinAddresses 
  \switchcolumn
\VicenteAddresses
\end{paracol}

\end{document}